\documentclass[12pt]{article}

\usepackage[a4paper]{geometry}
\usepackage{cite}
\usepackage[utf8]{inputenc}
\usepackage{amssymb}
\usepackage{amsmath}
\usepackage{amsthm}
\usepackage{empheq}
\usepackage{relsize}
\usepackage{siunitx}
\allowdisplaybreaks 

\usepackage{titlesec}
\usepackage{pgfplots}

\usepgfplotslibrary{fillbetween}
\pgfplotsset{samples=100}

\usepackage{calrsfs}
\DeclareMathAlphabet{\pazocal}{OMS}{zplm}{m}{n}

\usepackage{enumerate}
\usepackage{hyperref}
\usepackage{titlesec}
\theoremstyle{plain}

\newtheorem{thm}{Theorem}[section]

\newtheorem{remark}{Remark}[section]

\newcommand{\tb}{\textcolor{blue}}
\title{Sections and Chapters}
\numberwithin{equation}{section}   
\titleformat{\chapter}[display]
  {\bfseries\Large}
  {\filleft\MakeUppercase{\chaptertitlename} \Huge\thechapter}
  {1ex}
  {\titlerule\vspace{1ex}\filright}
  [\vspace{1ex}\titlerule]

\title{Liouville-type results for time-dependent stratified water flows over variable bottom in the $\beta$-plane approximation}
\author{Calin Iulian Martin$^\ast$ \\
\textit{Faculty of Mathematics, University of Vienna,}\\ 
\textit{Oskar-Morgenstern-Platz 1, 1090, Vienna, Austria}\\
{\small $^\ast$Email address: calin.martin@univie.ac.at}\\
}

\date{}

\begin{document}
	
\maketitle	

\def\simbolo#1#2{#1\dotfill{#2}}
\def\qed{\hbox to 0pt{}\hfill$\rlap{$\sqcap$}\sqcup$\medbreak}
\def\theequation{\arabic{section}.\arabic{equation}}
\def\thesection {\arabic{section}}

\begin{abstract}
We consider here time-dependent three-dimensional stratified geophysical water flows of finite depth over a variable bottom with a free surface and an interface (separating two layers of constant and different densities).
Under the assumption that the vorticity vectors in the two layers are constant, we prove that bounded solutions to the three-dimensional water waves equations in the $\beta$-plane approximation exist if and only if one of the horizontal components of the velocity, as well as its vertical component, are zero; the other horizontal component being constant. Moreover, the interface is flat, the free surface has a traveling character in the horizontal direction of the nonvanishing velocity component, being of general type in the other horizontal direction, and the pressure is hydrostatic in both layers. Unlike previous studies of three-dimensional flows with constant vorticity in each layer, we consider a non-flat bottom boundary and different constant vorticity vectors for the upper and lower layer.
\end{abstract}

\maketitle

\noindent
{\bf Keywords}: Time-dependent thee-dimensional gravity water flows, stratification, $\beta$-plane effects, variable bottom, piecewise constant vorticity.
\bigbreak

\noindent
{\bf Mathematics Subject Classification:} 35A01, 35Q35, 35R35, 76B15, 76B70.
\bigbreak

\section{Introduction}

Geophysical fluid dynamics (GFD) is the study of fluid motion characterized by the incorporation of Coriolis effects in the governing equations. The Coriolis force is a result of the Earth's rotation and plays a substantial role in the resulting dynamics. While the nonlinear GFD equations are able to capture a plethora of oceanic and atmospheric flows \cite{CJGAFD, CJaz, CJazAcc, cj1, cj2, cj3, cj4} their high level of difficulty greatly challenges the available mathematical techniques. 

The possible course of action is the recourse to simpler approximate models that are justified by oceanographical considerations. One of these approximations refers to the linearization of Coriolis forces in the tangent plane approximation, a procedure that-despite the spherical shape of the Earth-is valid due to the moderate spatial scale of the motion: the region occupied by the fluid can be approximated by a tangent plane and the linear term of the Taylor expansion captures the $\beta$-plane effect, cf. the discussions in Constantin \cite{CoJGeoR}, Cushman-Roisin \& Beckers \cite{Cush}, Gill \cite{Gill}, Pedlosky \cite{Ped}, Salmon \cite{Sal}. The paper \cite{CoJGeoR}, by Constantin, presented for the first time three-dimensional explicit and exact solutions (in Lagrangian coordinates) to the equatorial $\beta$-plane model. These solutions described equatorially trapped waves propagating eastward in a stratified inviscid flow.
 The latter $\beta$-plane model was modified to incorporate centripetal terms by Constantin \& Johnson \cite{CJaz} whereby the authors also established the existence of equatorial purely azimuthal solutions to the GFD equations in spherical, cylindrical and $\beta$-plane coordinates. A further significant extension was realized by Henry \cite{HenJFM16} who presented an exact and explicit solution (to the GFD equations in the $\beta$-plane approximation with Coriolis and centripetal forces) representing equatorially trapped waves propagating in the presence of a constant background current. Yet, another improvement in the realm of the GFD equations in the $\beta$-plane was obtained by Henry \cite{HenJDE17} and concerns the addition of a gravitational-correction term in the tangent plane approximation. 
Historically, this type of approximation was proposed by Rossby {\emph et al.} \cite{Ros} as a conceptual model for motion on a sphere. 
Recent results concerning flows in the $\beta$-plane approximation were obtained in \cite{Davi, Dur}. For subsequent solutions to the GFD equations concerning a variety of geophysical scenarios, we refer the reader to \cite{BasDSR, CoPhysOc2013, CoPhysOc2014, CJPhysFluids2017, CJJPO2019, HenMarARMA, dikNA, dikJMFM15, MarPoF21, MartinOc, AMJPhA, AMApplAna, AMQAM}.

In the quest to derive explicit and exact solutions to the GFD equations in the $\beta$-plane approximation with Coriolis and centripetal terms, we start from the assumption that the vorticity vector is constant in each of the two layers of 
the fluid domain which is assumed to be stratified: the water flow, bounded below by a bottom (varying) boundary and above by the free surface, is split by an interface into a layer adjacent to the bottom of some constant 
density (say $\rho$) which sits below the layer adjacent to the free surface of constant density $\tilde{\rho}<\rho$. Stratification is an important aspect in the ocean science: stratified layers act as a barrier to the mixing of water, which impacts the exchange of heat, carbon, oxygen and other nutrients, cf. \cite{Li}. The discontinuous stratification (of the type we consider here) gives rise to internal waves, an aspect that has attracted much attention lately from the perspective of exact solutions describing large-scale geophysical (and non-geophysical) flows \cite{Bas, CJGAFD, CJaz, CJazAcc, CJ_Oc, CICMP, EMM, EKLM, GeyQuiDCDS, GeyQuiCPAA, HenMatPisa, HenMarARMA, MarPoF21} or of qualitative studies of intricate features underlying the dynamics of coupled surface and internal waves, cf. \cite{HenVillJDE}.

The structural consequences of constant vorticity in water flows satisfying the three-dimensional equations were discussed in a handful of papers by Constantin \& Kartashova \cite{CK},  Constantin \cite{CEjmb}, Craig \cite{Cra}, Wahl\'{e}n \cite{Wah}, Stuhlmeier \cite{Stu}, Martin \cite{MarJmfm17, MarPoF, MarJFM, MarNonl19}: the main outcome is that occurrence of constant vorticity in a flow that satisfies the three-dimensional nonlinear governing equations is possible if and only if the flow is two-dimensional and if the vorticity vector has only one non-vanishing component that points in the horizontal direction orthogonal to the direction of wave propagation.
In particular, constant vorticity gives a good description of tidal currents; cf. \cite{Per}. These are the most regular and predictable currents, and on areas of the continental shelf and in many coastal inlets they are the most significant currents; cf. \cite{Jon}. 


After introducing the governing equations we state and prove a Liouville-type result in the context of a two-layer fluid domain bounded below by a bottom boundary $z=b(y)$ (for some given function $(x,y)\rightarrow b(y)$) and above by the free surface $z=\tilde{\eta}(x,y,t)$ (for some unknown function $\tilde{\eta}$), and split by an interface $z=\eta(x,y,t)$ (for some unknown function $\eta$) into two layers. We prove that 
 if the vorticity vectors in the two layers are constant then bounded solutions to the three-dimensional water waves equations in the $\beta$-plane approximation (with the associated boundary conditions) exist if and only if one of the horizontal components of the velocity, as well as its vertical component, are zero; the other horizontal component being constant. Moreover, the interface is flat, the free surface has a traveling character in the horizontal direction of the nonvanishing velocity component, being of general type in the other horizontal direction, and the pressure is hydrostatic in both layers. \tb{We would like to underline that allowing for an extra $x$-dependence in the bottom function $b$, leads only to unbounded solutions, cf. Remark \ref{x_dep}}.

We would also like to note that, unlike previous Liouville-type results, \cite{CK, CEjmb, MarJmfm17, MarPoF, MarJFM, MarNonl19, MarLiouville, Stu, wheel, Wah}, our analysis here does not require the bottom boundary to be flat.
It is also worth to note that results similar in outcome, but for the Euler- or Navier-Stokes equations without a free boundary, were obtained recently and relatively recently, cf. \cite{GigaCmp, GigaCpde}. The ethos of the previously mentioned studies is that under integrability conditions or conditions concerning the mean oscillation of the velocity field, the latter vanishes identically, or that it displays less complexity.  

\section{The three dimensional water wave problem in the $\beta$-plane approximation}\label{threedimeuler}
We choose to work in a rotating framework with the origin at a point on the Earth's surface which is approximated by a sphere of radius $R=6378$ km
and denote with $(x,y,z)$ the Cartesian coordinates where the spatial variable $x$ refers to the longitude, the variable $y$ to latitude, and the variable $z$ stands for the local vertical, cf. Fig. \ref{syst}.
\begin{figure}[!h]
\begin{center}
\includegraphics*[width=0.75\textwidth]{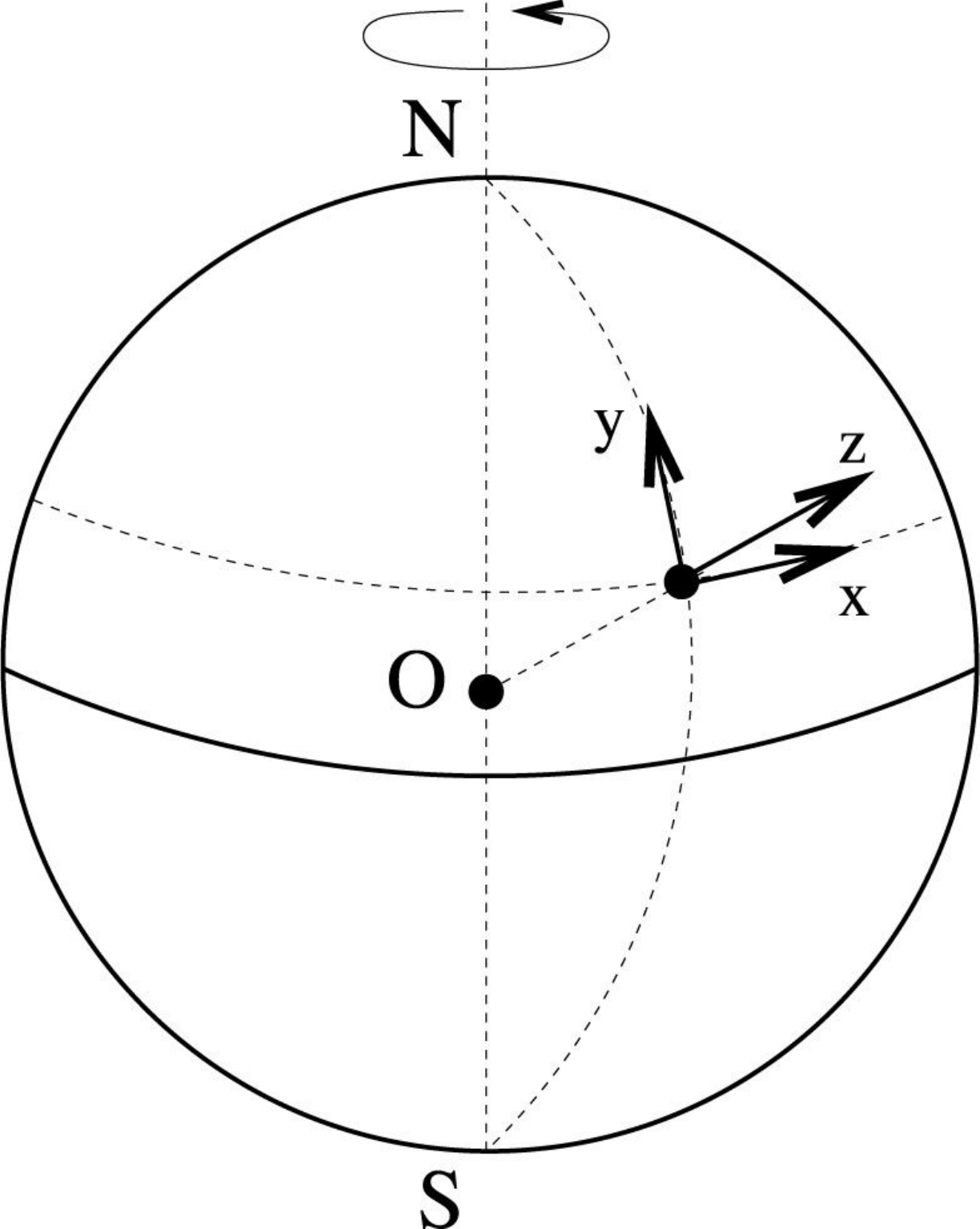}
\end{center}
\caption{The rotating frame of reference, with the x axis chosen horizontally due east, the y axis chosen horizontally due north, and the z axis chosen upward: x corresponds to longitude, y to latitude, and z to the local vertical.}
\label{syst}
\end{figure}

The fluid domain, bounded below by a bottom boundary $z=b(x,y)$ (for some differentiable function $(x,y)\rightarrow b(x,y)$) and above by the free surface $z=\tilde{\eta}(x,y,t)$ is split by an interface, denoted $z=\eta(x,y,t)$, into two layers: a layer adjacent to the bottom, written as $$D_{\eta}(t):=\{(x,y,z): b(x,y)\leq z\leq \eta(x,y,t)\},$$ which sits below a layer adjacent to the surface, written as 
$$D_{\eta,\tilde{\eta}}(t):=\{(x,y,z): \eta(x,y,t)\leq z\leq \tilde{\eta}(x,y,t)\}.$$
We denote with ${\bf u}(x,y,z,t)=(u(x,y,z,t),v(x,y,z,t),w(x,y,z,t))$ (respectively $\tilde{\mathbf{u}}(x,y,z,t)=(\tilde{u}(x,y,z,t),\tilde{v}(x,y,z,t), \tilde{w}(x,y,z,t))$) the velocity field in $D_{\eta}(t)$ (respectively in $D_{\eta,\tilde{\eta}}(t)$),
 with $P(x,y,z,t)$ (respectively $\tilde{P}(x,y,z,t)$) the pressure in $D_{\eta}(t)$ (respectively  in $D_{\eta,\tilde{\eta}}$),
and with $g$ the gravitational acceleration. We denote with $\rho$ the density in the lower layer $D_{\eta}$ 
and with $\tilde{\rho}$ the density in the upper layer $D_{\eta,\tilde{\eta}}$.
Then the motion of incompressible and inviscid three-dimensional water flows in the $\beta$-plane approximation near the Equator, (e.g. Constantin \cite{CoPhysOc2013}, Constantin \& Johnson \cite{CJaz} and Dellar \cite{Del}), is governed in 
$D_{\eta}$ by the Euler equations
\begin{equation}\label{Euler_eq}
 \begin{split}
  u_{t}+u u_{x}+v u_{y}+w u_{z}+2\omega w-\beta yv &=  -\frac{P_x}{\rho}\\
  v_{t}+u v_{x}+v v_{y}+w v_{z} +\beta y u+\omega^2 y&=-\frac{P_y}{\rho}\\
 w_{t}+u w_{x}+v w_{y}+ w w_{z}-2\omega u -\omega^2 R &= -\frac{P_z}{\rho}-g
  \end{split}
\end{equation}
and by the incopressibility condition
\\
\begin{equation}\label{mass_cons}
 u_{x}+v_{y}+w_{z}=0.
\end{equation}
Above, $t$ denotes the time variable, $\omega=7.29\cdot 10^{-5}$ rad $s^{-1}$ is the rotational speed of Earth round the polar axis toward east and $\beta:=\frac{2\omega}{R}$.
\begin{remark}
The $\beta$-plane effect, captured by the quantity $\beta y$, appears in the first and second equation in \eqref{Euler_eq}, and is a result of linearizing the Coriolis force in the tangent plane approximation. In spite of the spherical shape of the Earth, the before-mentioned linearization procedure is legitimate due to the moderate spatial scale of the motion, cf. the discussions in Cushman-Roisin \& Beckers \cite{Cush} and Constantin \cite{CoJGeoR}. We would like to note
that the $\beta$-plane equations \eqref{Euler_eq} represent a consistent approximation to the governing equations only near the Equator, cf. Dellar \cite{Del}.
\end{remark}

Likewise, in $D_{\eta,\tilde{\eta}}$ the flow motion obeys the equations
\begin{equation}\label{Euler_eq_tilde}
 \begin{split}
  \tilde{u}_{t}+\tilde{u} \tilde{u}_{x}+\tilde{v} \tilde{u}_{y}+\tilde{w} \tilde{u}_{z}+2\omega \tilde{w}-\beta y\tilde{v} &=  -\frac{\tilde{P}_x}{\tilde{\rho}}\\
  \tilde{v}_{t}+\tilde{u} \tilde{v}_{x}+\tilde{v} \tilde{v}_{y}+\tilde{w}\tilde{v}_{z}+\beta y \tilde{u}+\omega^2 y & =-\frac{\tilde{P}_y}{\tilde{\rho}}\\
 \tilde{w}_{t}+\tilde{u}\tilde{ w}_{x}+\tilde{v}\tilde{ w}_{y}+ \tilde{w}\tilde{ w}_{z}-2\omega\tilde{u}-\omega^2 R & = -\frac{\tilde{P}_z}{\tilde{\rho}}-g
  \end{split}
\end{equation}
and 
\begin{equation}\label{mass_cons_up}
\tilde{u}_x+\tilde{v}_y+\tilde{w}_z=0.
\end{equation}
We specify now the boundary conditions which conclude the formulation of the water wave problem. We start with the kinematic boundary conditions, stating the impermeability of the boundaries: on the free surface $z=\tilde{\eta}(x,y,t)$ we require 
\begin{equation}\label{kin_fs}
\tilde{w}=\tilde{\eta}_t+\tilde{u}\tilde{\eta}_x+\tilde{v}\eta_y\quad {\rm on}\quad z=\tilde{\eta}(x,y,t)
\end{equation}
On the interface $z=\eta(x,y,t)$ the conditions 
\\
\begin{equation}\label{kin_int}
\begin{split}
\tilde{w}=\eta_t+\tilde{u}\eta_x+\tilde{v}\eta_y\quad {\rm on}\quad z=\eta(x,y,t)\\
 w=\eta_t +u\eta_x+v\eta_y \quad {\rm on}\quad z=\eta(x,y,t)
\end{split}
\end{equation}
\\
while on the bed it holds that
\begin{equation}\label{kin_bed}
w=ub_x+vb_y\quad{\rm on}\quad z=b(x,y).
\end{equation}
\\
The balance of forces at the interface is encoded in the continuity of the pressure across $z=\eta(x,y,t)$, that is
\begin{equation}\label{bal_int}
P(x,y,\eta(x,y,t),t)=\tilde{P}(x,y,\eta(x,y,t),t)\,\,\textrm{for all}\,\,x,y,t.
\end{equation}
Lastly, the dynamic boundary condition, states the continuity of the pressure across the free surface, that is, we require that
\\
\begin{equation}\label{dyn_bc}
 \tilde{P}=p\left(x+\frac{\omega R}{2}t,y\right)\quad{\rm on}\quad z=\tilde{\eta} (x,y,t),
\end{equation}
\\
for some given differentiable function $(X,Y)\rightarrow p(X,Y)$.\\
The local rotation in the flow is captured by the vorticity vector, defined as the curl of the velocity field, that is,
\begin{equation}\label{vort_def}
\begin{split}
 \Omega=(w_{y}-v_{z},u_{z}-w_{x},v_{x}-u_{y})=:(\Omega_1,\Omega_2,\Omega_3)\,\,{\rm in}\,\,D_{\eta},\\
\tilde{\Omega}=(\tilde{w}_y-\tilde{v}_z, \tilde{u}_z-\tilde{w}_x,\tilde{v}_x-\tilde{u}_y)=:(\tilde{\Omega}_1,\tilde{\Omega}_2,\tilde{\Omega}_3)\,\,{\rm in}\,\, D_{\eta,\tilde{\eta}}.
\end{split}
\end{equation}
\begin{remark}
The (piecewise) constant vorticity is instrumental in describing wave-current interactions in sheared flows \cite{Jon, Ligh, Per, ThomKlop}. For a comprehensive treatment of two-dimensional water flows with discontinuous vorticity 
(from the point of view of exact solutions describing waves of small and large amplitudes) we refer the reader to the work by Constantin \& Strauss \cite{CSArma2011}. However, the landscape of rotational three-dimensional water flows is much less understood. Nevertheless, it is known that  in three-dimensional flows the constant vorticity significantly steers the dimensionality of the velocity field and of the pressure \cite{CK, CEjmb, MarJmfm17, MarPoF, MarJFM, MarNonl19, Stu, Wah}. Compared with previous studies on rotational three-dimensional water flows, we move here one step further and assume that the vorticity vectors, $\Omega$ and $\tilde{\Omega}$, respectively, are constant. As we shall see, this assumption will have massive consequences on the velocity field.
\end{remark}
\noindent We are now ready to state the main result for whose proof we will rely on rather direct partial differential equations methods, unlike more sophisticated Hamiltonian tools (and other structure-preserving methods) used recently in the context of layered domains \cite{Brid, CIM, CICMP, Wei, Wei2, Wei3}, which are not known to be available in our three-dimensional setting of the $\beta$-plane.
\begin{thm}\label{main}
If the velocity field satisfies $(u,v,w)\in L^{\infty}(D_{\eta})^3$ and $(\tilde{u},\tilde{v},\tilde{w})\in L^{\infty}(D_{\eta,\tilde{\eta}})^3$ then the flow has constant vorticity $\Omega\in D_{\eta}$ and $\tilde{\Omega}\in D_{\eta,\tilde{\eta}}$ if and only if
$$u(x,y,z,t)=u(t),\,\,v(x,y,z,t)=0,\,\,w(x,y,z,t)=w(t),\,\,{\rm for}\,\,{\rm all}\,\,x,y,z,t,$$
$$\tilde{u}(x,y,z,t)=\tilde{u}(t),\,\,\tilde{v}(x,y,z,t)=0,\,\,\tilde{w}(x,y,z,t)=\tilde{w}(t),\,\,{\rm for}\,\,{\rm all}\,\,x,y,z,t,0,$$
that is, $u,\tilde{u},w,\tilde{w}$ depend only on $t$, while $v$ and $\tilde{v}$ vanish.

\noindent If, in addition, the bottom defining function $b$ depends only on $y$, and $P$ and $\tilde{P}$ are bounded in $\Omega$, and $\tilde{\Omega}$, respectively, then
$$u=\tilde{u}=-\frac{\omega R}{2}\,\,{\rm and}\,\,v=w=\tilde{v}=\tilde{w}=0.$$ 
If, moreover, $\rho\neq \tilde{\rho}$ then $$\eta(x,y,t)=0\,\,{\rm for}\,\,{\rm all}\,\, x,y,t,$$
and 
$$\tilde{\eta}(x,y,t)=-\frac{p\left(x+\frac{\omega R}{2}t,y\right)}{\tilde{\rho}g},$$
where $p$ is the given function from \eqref{dyn_bc}.
\end{thm}
\begin{proof}
Elininating the pressure between the three equations in \eqref{Euler_eq} yields the vorticity equation 
\begin{equation}\label{vort_eq}
\begin{split}
\Omega_1 u_x+(\Omega_2+2\omega)u_y+(\Omega_3+\beta y)u_z &=0,\\
\Omega_1 v_x+(\Omega_2+2\omega)v_y+(\Omega_3+\beta y)v_z &=0,\\
\Omega_1 w_x+(\Omega_2+2\omega)w_y+(\Omega_3+\beta y)w_z &=\beta v,\\
\end{split}
\end{equation}
cf. eg. \cite{Cobook, John, MB}.
Applying the operator $\Delta$ to the first equation above and using that $u_x$, $u_y$ and $u_z$ are harmonic functions we find that $\Delta(yu_z)=0$. The latter equation can be expanded as 
$$y\Delta u_z+2u_{yz}=0$$ which leads to 
\begin{equation}\label{u_yz}
u_{yz}\equiv 0.
\end{equation}
Following the same line of proof we conclude that 
\begin{equation}\label{vw_yz}
v_{yz}(x,y,z,t)=w_{yz}(x,y,z,t)=0
\end{equation} 
at all points $(x,y,z)$ of $D_{\eta}$. Utilizing the definitions of $\Omega_1$ and $\Omega_2$ and recalling \eqref{vw_yz} we obtain
\begin{equation}\label{wy_const}
w_{yx}=w_{yy}=w_{yz}=0,
\end{equation}
which proves that $w_y$ is constant throughout $D_{\eta}$. That is, there is a differentiable function $t\rightarrow f(t)$ such $$w_y=f(t),$$ from which we infer the existence of another differentiable function $t\rightarrow g(t)$ such that 
\begin{equation}\label{v_z}
v_z=g(t).
\end{equation}
We apply now the operator of differentiation with respect to $y$ in the vorticity equation \eqref{vort_eq} and obtain, via \eqref{wy_const} that 
\begin{equation}
w_z=v_y\,\,{\rm within}\,\, D_{\eta}.
\end{equation}
We now see from above that the equalities 
\begin{equation}\label{w_zz_v_yy}
w_{zz}=v_{yz}=0\quad{\rm and}\quad v_{yy}=w_{yz}=0
\end{equation}
hold at all points of $D_{\eta}$.
Differentiating the equation of mass conservation \eqref{mass_cons} with respect to $z$ we get that $u_{xz}+v_{yz}+w_{zz}=0$, and so we have via \eqref{w_zz_v_yy} that 
\begin{equation}
u_{xz}\equiv 0,
\end{equation}
while from the differentiation of \eqref{mass_cons} with respect to $y$ we infer that 
\begin{equation}
u_{xy}\equiv 0.
\end{equation}
Differentiating with respect to $z$ in the first equation of \eqref{vort_eq} and recalling that $u_{xz}=u_{yz}=0$ we see that $(\Omega_3+\beta y)u_{zz}=$ for all $(x,y,z)$ in $D_{\eta}$. Hence, 
\begin{equation}\label{u_zz}
u_{zz}=w_{xz}=0
\end{equation}
at all points of $D_{\eta}$. We claim now that 
\begin{equation}\label{u_xx}
u_{xx}\equiv 0\quad {\rm within}\quad D_{\eta}.
\end{equation}
To prove \eqref{u_xx} we distinguish two cases as follows:
\begin{enumerate}
\item [ (i) ] $\Omega_1=0$. Then the second equation in \eqref{vort_eq} becomes 
$$(\Omega_2+2\omega)v_y +(\Omega_3+\beta y)v_z=0.$$
After a differentiation with respect to $x$ in the previous equation and accounting also for \eqref{v_z}, we obtain that 
\begin{equation}
v_{xy}=u_{yy}=0\quad {\rm within}\quad D_{\eta},
\end{equation}
so that the harmonicity of $u$ and relation \eqref{u_zz} provides the claim in \eqref{u_xx}.
\item [ (ii) ] $\Omega_1\neq 0$. In this case we differentiate the first equation of \eqref{vort_eq} with respect to $x$ and avail of $u_{xy}=u_{xz}=0$ to conclude the correctness of the claim in \eqref{u_xx}.
\end{enumerate}
We can infer now from \eqref{u_xx}, \eqref{u_zz} and the harmonicity of $u$ that 
\begin{equation}
u_{yy}=0\quad {\rm within}\quad D_{\eta}.
\end{equation}
An inspection of the previous considerations shows that $\nabla u$, $\nabla v$ and $\nabla w$ are (vectorial) functions that depend only on $t$. Corroborating the latter with the boundedness of $u,v,w$ yields that $u,v,w$ are functions that 
depend only on $t$. The latter implies now that the spatial gradients of $u,v$ and $w$ vanish identically within $D_{\eta}$. This last information yields via \eqref{vort_eq} that 
\begin{equation}\label{v=0}
v=0\quad{\rm  throughout}\quad D_{\eta}.
\end{equation}
Since $w$ is constant within $D_{\eta}$ we conclude via the bottom condition \eqref{kin_bed} that $w=0$ throughout the lower layer $D_{\eta}$.
Consequently, the Euler equations are now written as
\begin{equation}
\begin{split}
P_x&=-\rho u'(t),\\
P_y&=-\rho(\beta u(t)+\omega^2)y,\\
P_z&=\rho(2\omega u(t)-g+\omega^2 R),
\end{split}
\end{equation}
so that 
\begin{equation}
P(x,y,z,t)=-\rho \left[u'(t)x+(\beta u(t)+\omega^2)\frac{y^2}{2}+(g-2\omega u(t)-\omega^2 R )z\right].
\end{equation}
The boundedness of the pressure implies now that 
$$u(t)=-\frac{\omega^2}{\beta}=-\frac{\omega R}{2}\quad {\rm for}\,\,{\rm all}\quad t,$$
and therefore for all $x,y,t$ and all $-d\leq z\leq \eta(x,y,t)$ it holds
\begin{equation}
P(x,y,z,t)=-\rho g z.
\end{equation}
Utilizing the vorticity equation for the domain $D_{\eta,\tilde{\eta}}$ and employing an argument analogous to the one that led to the constancy of $u,v,w$ 
we conclude that $\tilde{u}$ and $\tilde{w}$ are constants of the time $t$ and $\tilde{v}$ vanishes throughout $D_{\eta,\tilde{\eta}}$. Hence, the Euler equations in the upper layer are written as
\begin{equation}
\begin{split}
\tilde{P}_x&=-\tilde{\rho} \left( 2\omega \tilde{w}(t) +\tilde{u}'(t)\right),\\
\tilde{P}_y&=-\tilde{\rho}(\beta \tilde{u}(t)+\omega^2)y,\\
\tilde{P}_z&=\tilde{\rho}(2\omega \tilde{u}(t)-g+\omega^2 R),
\end{split}
\end{equation}
which yields that the pressure in the upper layer $D_{\eta,\tilde{\eta}}$ is given as 
\begin{equation}
\tilde{P}(x,y,z,t)=-\tilde{\rho}\left[ \left( 2\omega \tilde{w}(t) +\tilde{u}'(t)\right)x+(\beta \tilde{u}(t)+\omega^2)\frac{y^2}{2}-(2\omega \tilde{u}(t)-g+\omega^2 R)z\right].
\end{equation}
Owing to the boundedness of $\tilde{P}$ we infer from the previous formula that 
\begin{equation}
\tilde{u}(t)=-\frac{\omega^2}{\beta}=-\frac{\omega R}{2}\quad {\rm and}\quad \tilde{w}(t)=0\quad{\rm for}\,\,{\rm all}\quad t,
\end{equation}
and for all $x,y,z,t$ such that $\eta(x,y,t)\leq z\leq \tilde{\eta}(x,y,t)$ it holds
\begin{equation}
\tilde{P}(x,y,z,t)=-\tilde{\rho} g z.
\end{equation}
Applying now the balance of forces at the interface we obtain (availing also of $\rho\neq \tilde{\rho}$) that $\eta(x,y,t)=0$ for all $x,y,t$. With this finding we immediately see that the kinematic conditions on the interface \eqref{kin_int}
are also verified.\\
The kinematic condition on the free surface \eqref{kin_fs} becomes $$\tilde{\eta}_t-\frac{\omega R}{2}\tilde{\eta}_x=0,$$ whose general solution is given as 
$\tilde{\eta}(x,y,t)=F\left(x+\frac{\omega R}{2}t,y\right)$ for some function $(X,Y)\mapsto F(X,Y)$. But, from the dynamic boundary condition \eqref{dyn_bc} we have that 
$F=-\frac{p}{\tilde{\rho} g}$.
\end{proof}
\noindent We conclude by revealing the reason for choosing the bottom defining function $b$ to depend only on $y$. 
\begin{remark}\label{x_dep}
Maintaining the hypotheses from Theorem \ref{main} and assuming that the bottom defining function presents a most general dependence $(x,y)\rightarrow b(x,y)$ with $b_x\neq 0$ we obtain that no solution with bounded pressure in the lower layer $D_{\eta}$ exists. Indeed, proceeding like in the proof of Theorem \ref{main} we notice that the arguments and the conclusions therein hold verbatim also in this scenario (with a more general bottom) until (and including) formula \eqref{v=0}. Then, utilizing the bottom condition $w=ub_x+vb_y$ on $z=b(x,y)$ we get $w(t)=u(t)b_x(x,y)$ for all $x,y$ and for all $t$. Assuming, {\emph ad absurdum} that there is $t_0$ such that $u(t_0)\neq 0$, we obtain from the previous equality that $b_{xx}=b_{yy}=0$. Since $b$ is bounded, it follows that $b_x(x,y)=0$ for all $x,y$, which is a contradiction with the hypothesis that $b_x\neq 0$. Thus, $u(t)=0$ for all $t$. From $w(t)=u(t)b_x$ it follows that also $w(t)=0$ for all $t$. The governing equations now yield that $P_y=-\rho \omega^2 y$ which clearly impedes the boundedness of $P$.
\end{remark}
\section{Conclusion}
We explored here the impact of (piecewise) constant vorticity on the dimension reduction of the velocity field in water flows satisfying the three-dimensional governing equations. This is part of a broader research agenda
 \cite{Bas, CEjmb, CoPhysOc2013, CoPhysOc2014, CJPhysFluids2017, MarPoF, MarJFM, MarNonl19, Stu, Wah} that aims to deepen the analytical understanding of the three-dimensional water waves equations with vorticity. 


In addition to the previous aspects concerning the vorticity, our analysis takes into account the presence of geophysical effects in the form of the (equatorial) $\beta$-plane approximation: this procedure consists in linearizing the Coriolis force in the tangent plane approximation, this course of action being justified by the moderate scale of motion. Unlike the $f$-plane approximation, the $\beta$-plane approximation \cite{Ped}
displays the essential characteristic that the Coriolis parameter is not constant in space. This is a feature that not only makes the equations of motion more tractable but also allows the study of many phenomena in the atmosphere and ocean; in particular,  Rossby waves, the most important type of waves for large-scale atmospheric and oceanic dynamics, depend on the variation of the Coriolis parameter as a restoring force cf. \cite{Hol}.

Under the assumptions (on the vorticity) stated before, we showed that the bounded solutions have zero vertical and one horizontal velocity components, while the other horizontal component is constant. 
We have also proved that the interface is flat, the free surface has a traveling character in the horizontal direction of the nonvanishing velocity component, being of general type in the other horizontal direction, and the pressure is hydrostatic in both layers. 

Different from other scenarios investigated before \cite{CK, CEjmb, MarJmfm17, MarPoF, MarJFM, MarNonl19, MarLiouville, Stu, Wah}, we allow a varying bottom boundary. A further improvement, when compared with previous studies dealing with three-dimensional flows with piecewise constant vorticity \cite{MarLiouville}, the vorticity vectors corresponding to the upper and lower layer, respectively, need not be parallel.

Our conclusion (regarding the dimension reduction of the flow) is reinforced in the study by Xia and Francois \cite{Xia}, which shows that large-scale structures in thick fluid layers can suppress vertical eddies and reinforce the planarity of the flow. In connection with planar flows we would also like to mention the result on the geometric structure of flows by Sun \cite{Sun}.

\section{Appendix}\label{App}
To justify the consideration of the system \eqref{Euler_eq} and for the sake of self-containedness, we recall in this section a derivation of the governing equations for the $\beta$-plane, as presented by Constantin \& Johnson \cite{CJaz}. 

We start from the setting of cylindrical coordinates in which the equator is replaced by a line parallel to the $x$ axis: the equator is ``straightened'' and the body of the sphere is represented by a circular disc, which is mapped out by the corresponding polar coordinates $(x,\theta, z)$. Here, $x$ stands for the azimuthal direction and points from west to east, the direction of increasing $\theta$ is from north to south, and $z$ represents the local vertical. Then, the governing equations for water wave propagation written in cylindrical coordinates $(x, \theta, z)$ in a rotating framework are the momentum conservation equations 
\begin{equation}\label{cyl}
\begin{split}
u_t+uu_x+\frac{vu_{\theta}}{R+z}+wu_z+2\omega (w\cos\theta-v\sin\theta)&=-\frac{p_x}{\rho},\\
v_t+uv_x+\frac{vv_{\theta}}{R+z}+\frac{wv_z}{R+z}+2\omega u\sin\theta+(R+z)\omega^2\sin\theta\cos\theta&=-\frac{p_{\theta}}{\rho(R+z)},\\
w_t+uw_x+\frac{vw_{\theta}}{R+z}+ww_z-\frac{v^2}{R+z}-2\omega u\cos\theta-(R+z)\omega^2\cos^2\theta&=-\frac{p_z}{\rho}-g,
\end{split}
\end{equation}
and the equation of mass conservation 
\begin{equation}
u_x+\frac{v_{\theta}}{R+z}+\frac{\big((R+z)w\big)_z}{R+z}=0,
\end{equation}
where $u,v,w$ are the components of the velocity field corresponding to the $x$, $\theta$ and $z$ variable, respectively.

Since we are interested in the behavior of the flow close to the equator we will confine the discussion to small $\theta$.  Our considerations also take into account the fact that the radius of Earth, with respect to the depth of the oceans, is extremely large. Thus, it is justified to perform the approximations 
\begin{equation}
\sin\theta\approx \theta,\quad \frac{z}{R}\approx 0\quad.
\end{equation}
Setting $y:=R\theta$ and disregarding the terms smaller than $O(\theta)$ in the expansion of the trigonometric functions, the equations of momentum conservation \eqref{cyl} become
\begin{equation}
\begin{split}
u_t+uu_x+vu_y+wu_z+2\omega\left(w-\frac{y}{R}v\right)&=-\frac{p_x}{\rho},\\
v_t+uv_x+vv_y+wv_z+2\omega\frac{y}{R}u+\omega^2 y&=-\frac{p_y}{\rho},\\
w_t+uw_x+vw_y+ww_z-2\omega u-R\omega^2&=-\frac{p_z}{\rho}-g,
\end{split}
\end{equation}
while the equation of mass conservation acquires the form
\begin{equation}
u_x+v_y+w_z=0,
\end{equation}
which are precisely the equations \eqref{Euler_eq} and \eqref{mass_cons}.

\noindent {\bf Acknowledgements}: 
The support of the Austrian Science Fund (FWF) under research grant P 33107 N is gratefully acknowledged. The author is indebted to the referees whose comments and suggestions have improved the quality of the manuscript.

\noindent {\bf Data Availability Statement}: No data is associated to this manuscript.

\end{document}